\documentclass[12pt]{amsart}
\usepackage{amscd, amssymb}
\input{xy}
\xyoption{all}

\newtheorem{Thm}{Theorem}[section]

\newtheorem{Prop}[Thm]{Proposition}
\newtheorem{Def/Thm}[Thm]{Definition/Theorem}

\newtheorem{Lemma}[Thm]{Lemma}
\theoremstyle{remark}
\newtheorem{Def}[Thm]{Definition}
\newtheorem{Rmk}[Thm]{Remark}

\numberwithin{equation}{section}
\newcommand{\ti }{\times}
\newcommand{\ot }{\otimes}
\newcommand{\ra }{\rightarrow}

\newcommand{\Hom }{{\mathrm{Hom}}}

\newcommand{\tr }{{\mathrm{tr}}}

\newcommand{\rank }{{\mathrm{rank}}}
\newcommand{\cO}{{\mathcal{O}}}

\newcommand{\cV}{{\mathcal{V}}}

\newcommand{\cT}{{\mathcal{T}}}

\newcommand{\NN}{{\mathbb N}}

\newcommand{\PP }{{\mathbb P}}

\newcommand{\CC }{{\mathbb C}}
\newcommand{\ZZ }{{\mathbb Z}}

\newcommand{\XXX}{{X/\!\!/\!\!/G}}

\newcommand{\ka }{{\alpha}}

\newcommand{\kl }{{\lambda}}

\newcommand{\Tot}{\mathrm{Tot}}
\newcommand{\affW}{W/\!\!/_{\!0} G}
\newcommand{\bM}{{\bf M}}

\newcommand{\Rep}{\mathrm{Rep}_{C}(\overline{Q}, {\bf M}, K)}
\newcommand{\Repkl}{\mathrm{Rep}_{\lambda, C}(\overline{Q}, {\bf M}, K)}
\newcommand{\Repklo}{\mathrm{Rep}_{\lambda=0, C}(\overline{Q}, {\bf M}, K)}

\title[Twisted Quasimaps]{Stable quasimaps to holomorphic symplectic quotients}

\author[B. Kim]{Bumsig Kim}
\address{School of Mathematics \\
Korea Institute for Advanced Study \\ 85
Hoegiro, Dongdaemun-gu, Seoul, 130-722 \\
Republic of Korea}
\email{bumsig@kias.re.kr}



\begin{document}

\begin{abstract}
We study the moduli space of twisted quasimaps from a fixed smooth projective curve to a Nakajima's quiver variety
and the moduli space of $\delta$-stable framed twisted quiver bundles with moment map relations.
We show that they carry symmetric obstruction theories and when $\delta$ is large enough, they exactly coincide.
These results generalize works by D.E. Diaconescu \cite{Dia} about the ADHM quiver, in the framework of the
quasimap theory of I. Ciocan-Fontanine, D. Maulik and the author \cite{CK, CKM}.
\end{abstract}

\maketitle

\section{Introduction}
There are, so far, two classes of moduli examples which naturally carry symmetric obstruction theories:
\begin{itemize}
\item Moduli of stable objects in the abelian category of coherent sheaves on a Calabi-Yau threefold (\cite{Th, PT}).
\item Moduli of stable objects in the abelian category of representations of a quiver with relations given by a superpotential (\cite{Sz2}).
\end{itemize}

In this paper, we add one more such class: 
\begin{itemize}
\item Moduli of stable objects in the abelian category of coherent $\bf M$-twisted
quiver sheaves on a projective smooth curve $C$ associated to double quivers
with moment map relations. Here $\bM$ is a certain collection of pairs of line bundles on $C$ such that the product of each pair
is isomorphic to the canonical sheaf $\omega _C$. 
\end{itemize}

It arises as a curve counting on a Nakajima's quiver variety. 
If $C$ is an elliptic curve and $Y$ is a holomorphic symplectic quasi-projective variety, 
then the canonical perfect obstruction theory on the Hom scheme $\mathrm{Hom}(C, Y)$ of
morphisms from $C$ to $Y$
is symmetric since for a map $f:C\ra Y$,
\begin{eqnarray*}  \mathrm{ob} (f) ^* := H^1(C, f^*\mathcal{T}_Y)^* &\cong& H^0(C, f^*\Omega _Y \ot \omega _C) \\
             &\cong& H^0(C, f^*\mathcal{T}_Y) = \mathrm{def} (f), \end{eqnarray*}
where $\mathrm{ob} (f)$ is the obstruction space at $f$ and $\mathrm{def}(f)$ is the 1st order deformation space at $f$.

When $Y$ is given by a holomorphic symplectic quotient of an affine algebraic variety $X$ by a reductive complex algebraic group $G$ action, 
we can apply the quasimap construction of \cite{CK, CKM}
in order to compactify  $\mathrm{Hom}(C, Y)$ relatively over the affine quotient of $X$. 
The advantage of the construction compared to the stable map construction is that
we can keep the fixed domain curve $C$ so that the natural extension of the obstruction theory still remains symmetric (Proposition
\ref{sym}).
When the affine variety $X$ carries a torus $T$ action commuting with the $G$ action such that the fixed locus $(\XXX )^{T}$ is proper over $\CC$, then
the induced $T$-fixed locus on the moduli space of stable quasimaps is also proper over $\CC$. 
Therefore, in this case one can obtain
well-defined localization residue \lq\lq invariants" by the virtual localization \cite{GP}.

If $Y$ is a Nakajima's quiver variety, a quasimap is nothing but a quiver bundle on $C$.  
Using the idea of twisted quiver bundles,
we obtain the corresponding notion of twisted quasimap and allow any genus of $C$. We show that the moduli space of stable 
$\bf M$-twisted quasimaps carries
a symmetric obstruction theory (Theorem \ref{twisted-symmetric}) and the stability as quasimaps coincides with 
the $\delta$-slope stability of quiver bundles for $\delta \gg 0$ (Proposition \ref{asym}).
In the case of ADHM quiver, these facts have been shown by Diaconescu in \cite{Dia1, Dia}, 
which is, together with the quasimap construction \cite{CK, CKM}, the main source of inspiration for this work.   
We also show that with respect to any $\delta$-slope stability, the moduli space of stable twisted quiver bundles
carries a canonical symmetric obstruction theory (Theorem \ref{bd}), following \cite{Dia1, Dia}. 

The typical examples for double quivers  in our framework 
are ADHM quiver (\cite{Dia1, Dia, CDP1, CDP2})
and the framed ADE quivers. The wall-crossing phenomena are studied in \cite{Dia, CDP1, CDP2, KL} 
which show that the conditions corresponding to (a) (the moduli stack is analytic-locally 
the critical locus of a holomorphic function on a smooth complex domain) and (b) (the Euler-like form is numerical)
 in \S 1.5 \cite{JS} hold in our setting. 

\medskip
\noindent{\bf Acknowledgements.} The author thanks Kurak Chung,  Ionu\c t Ciocan-Fontanine, Duiliu-Emanuel Diaconescu,
Sukmoon Huh, Byungheup Jun, Hoil Kim,  Hwayoung Lee, and Davesh  Maulik  for useful discussions; 
Kai Behrend and Bal\'azs Szendr\"oi for pointing out references; and the referee for generous comments to improve
the presentation. 
This work is financially supported by NRF-2007-0093859.

\section{Holomorphic symplectic quotients}
\subsection{Symplectic quotients}
We set up holomorphic symplectic quotients suitable for our purpose.

Let $X$ be a smooth affine algebraic variety over $\CC$ equipped 
with a holomorphic (algebraic) symplectic form
$$\omega : \cT_X\ot \cT_X \ra \mathcal{O}_X.$$
Suppose a connected reductive complex algebraic group $G$ acts on $X$ as a hamiltonian action, i.e., 
the action preserves $\omega$ and there is
a so-called {\em complex} moment map $\mu: X\ra \mathfrak{g} ^*$ where
$\mathfrak{g}$ denotes the Lie algebra of $G$.
This complex moment map $\mu$ is, by definition, a $G$-equivariant algebraic morphism
such that 
for every tangent vector  $\xi$ of $X$ and every $g\in \mathfrak{g}$, the equation 
\begin{equation}\label{moment} \langle d\mu (\xi ), g\rangle = \omega (d\ka (g), \xi ) \end{equation}
holds
where $d\ka$ is the derivative of the map $\ka : G\ra \mathrm{Aut}(X)$ induced from the $G$-action on $X$. Here
$G$-action  on $\mathfrak{g}^*$ is the coadjoint representation $\mathrm{Ad}^*$.

Choose an $\mathrm{Ad}^*$-invariant element $\lambda$ in $\mathfrak{g}^*$.
Denote $\mu ^{-1}(\lambda)$ by $W$. Suppose $ W \hookrightarrow X$ is a regular imbedding
of codimension $\dim \mathfrak{g}$.
For a character $\chi \in \Hom (G, \CC ^\ti )$, let  $ \CC _{\chi}$ 
be  the $G$-representation space $\CC$ associated to  $\chi$
and let $L$ be the linearization $W\times \CC _{\chi}$. The {\em holomorphic symplectic quotient} $X/\!\!/\!\!/_{\!\lambda, \chi}G$ is defined to be a GIT quotient 
\[ W/\!\!/_{\! \chi} G := \mathbf{Proj} (\bigoplus _{k\ge 0} H^0( W, L^k)^G ) \]
(see \cite[\S 4.5]{Gin} and references therein for Nakajima's quiver variety case).

\begin{Def}
A point $p$ of $W$ is called {\em semistable} with respect to $\chi$ if there is $s\in H^0(W, L^k)^G$ for some $k >0$ such
that $s(p)\ne 0$.
The semistable point $p$ is called {\em stable} with respect to $\chi$  if the stabilizer $G_p$ is finite 
and the action of $G$ on $\{ q : s(q)\ne 0 \}$ is closed (i.e., every orbit is closed in $\{ q : s(q)\ne 0 \}$).
\end{Def}


The stable (resp. semistable) locus is the open subset of $W$ consisting of all stable (resp. semistable) points.
According to the proof of Theorem 1.10 of \cite{Mu}, 
the GIT quotient is the categorical quotient of the semistable locus $W^{ss}$
and it contains as an open subset the geometrical quotient of the stable locus $W^s$ (see also 
\cite[Definition 2.2.3, Theorem 2.2.4]{Gin}).

\medskip

The equation (\ref{moment}) shows that
the stabilizer of a point $p\in X$ is a finite group if and only if
$d\mu _{|_{T_pX}} :T_pX \ra \mathfrak{g}^*$ is surjective if and only if $p$ is a regular point of 
the moment map $\mu$.
Hence the stack quotient $[W ^s /G]$ of the stable locus $W^s$
is a holomorphic symplectic smooth DM stack.

For a 1-parameter
subgroup, i.e., a homomorphism $\kappa : \CC ^\times \ra G$, denote by $\langle \chi , \kappa \rangle$
the exponent $m$ of $t^m = \chi (\kappa (t))$, $t\in \CC ^\times$.
There is a numerical criterion for semistable and stable points due to King \cite{King}.
This  will be used in Section \ref{King}.


\begin{Prop}\label{NC} {\em (\cite{King})} 
A point $p$ is semistable (resp. stable)  if \[ \langle \chi , \kappa \rangle \ge 0, \
{\text (resp.  }\ \langle \chi , \kappa \rangle > 0 \text{)} \] for any  nontrivial 1-parameter
subgroup $\kappa$ for which the limit of $\kappa (t) \cdot p $ as $t\ra 0$ exists.
\end{Prop}
\begin{proof}
The $G$-action on $X$ is linear in the sense that
there is a $G$-equivariant closed embedding of $X$ into a $G$-linear space $V$ (see e.g. \cite{Le}, page 94).
Since $G$ is geometrically reductive, the semistability of $p$ is equivalent to the fact that  the  closure 
of the orbit $G\cdot (p, 1)$  in $W\times \CC _{-\chi}$ for $(p,1)\in W\times \CC _{-\chi}$ is disjoint from the zero section $W\times \{0\}$ as in \cite[Lemma 2.2]{King}.
This translation shows that a point $p$ of $W$ is (semi)stable if and only if it is 
(semi)stable  as a point of the $G$-space $V$ with respect $\chi$.
Now the proof  immediately follows from \cite[Proposition 2.5]{King}.
\end{proof}

\begin{Rmk} In the above, the sign convention is  correct since 
$(g\cdot s) (p) = \chi (g)^l s (g^{-1}\cdot p)$ for $s\in  H^0( W, L^l)$.
\end{Rmk}


\subsection{Symmetry}
The derivatives of $\ka$ and $\mu$ together give rise to a
commutative diagram
\[\begin{CD} 0 @>>> \cO_X\ot \mathfrak{g}  @>{d\ka}>>  
   \cT _X  @>{d\mu}>>   \cO _X   \ot  \mathfrak{g} ^*  @>>> 0  \\
@VVV @V{\mathrm{id}}VV  @V{\omega }V{=-\omega ^\vee}V @VV{\mathrm{-id}}V @VVV\\
 0@>>> \cO _X \ot  \mathfrak{g}   @>>{d\mu ^\vee}>          \Omega _X @>>{d\ka ^\vee}>   \cO _X  \ot \mathfrak{g} ^* @>>> 0 \end{CD}
 \] of $G$-sheaves due to equation (\ref{moment}).
 Here the $G$-action on $\mathfrak{g}$ is the adjoint representation $\mathrm{Ad}$.
The first horizontal line of the diagram is a $G$-equivariant  complex after the restriction 
to $W$. 
Note that the 3-term perfect $G$-equivariant complex 
\[ \mathbb{T}_{[W/G]} :=\left[\cO _X \ot \mathfrak{g}  \stackrel{d\ka}{\longrightarrow}  
   \cT _X  \stackrel{d\mu}{\longrightarrow}      \cO _X    \ot \mathfrak{g} ^*  \right] _{|_{W}} \]  concentrated in degrees $-1, 0, 1$
  is nothing but (the pullback of) the tangent complex $\mathbb{T}_{[W/G]}$ of the quotient stack $[W/G]$.
Note that the commutative diagram shows that 
\begin{equation}\label{Fsym} \mathbb{T}_{[W/G]}^\vee \cong \mathbb{T}_{[W/G]} \end{equation} as complexes of $G$-sheaves.


\section{Stable Quasimaps}

 The notion of stable quasimaps appeared for a
compactification of maps from a smooth projective curve $C$ to a suitable GIT quotient with no strictly semistable points.
We first begin with conventions and definitions.
 
 \begin{Def}
By a {\em principal $G$-bundle} $\pi : P\ra Y $ on a scheme $Y$, we mean a scheme $P$ 
with a free {\em left} $G$-action which is \'etale locally trivial, i.e, there is an \'etale surjective morphism
from a scheme $Y'$ to $Y$ making the pullback $P\times _Y Y'$ isomorphic to $G\times Y'$ as a $G$-space over $Y'$.
By a {\em morphism between two principal bundles} on $Y$, we mean a $G$-equivariant morphism over $Y$.
\end{Def}

If $P$ is a principal bundle with right action as usual, then we will consider it as a principal bundle with
left action through $G\ra G$, $g\mapsto g^{-1}$.

\begin{Def} Define a {\em degree class}  $\beta$ of $P$ as the homomorphism
 \[  \beta  : \mathrm{Hom}(G,\ZZ)  \ra  \ZZ, \ \ 
                                                          \theta \mapsto  \deg (P\times _G \CC _\theta ).  \]    
\end{Def}

 
\subsection{Quasimaps}

We are interested in 1-morphisms $[u]$ from $C$ to the quotient stack $[W/G]$.
Such morphisms correspond to pairs $(P, u)$ of 
principal $G$-bundles $P$
on $C$ and sections $u$ of fiber bundle $P\times _G W$ on $C$.
For such a pair $(P, u)$, $[u]$ is the composite of
\[\xymatrix{ C \ar[r]^-{u} &  P\times _G W =[(P\times W)/G] \ar[r]^{\ \ \ \ \ \ \ \ \ \ \ \mathrm{pr}_2}  & [W/G] },   \]
where $\mathrm{pr}_2$ is the projection to the second factor.

\begin{Def} (\cite{CK, CKM}) In this paper, 
a pair $$(P, u), \text{ i.e., }  [u] : C\ra [W/G]$$ is called a {\em quasimap} to $[W^s/G]$.
The pair is called $stable$ with a {\em rigid} domain curve $C$
if the preimage $u^{-1} (P\times _G W^{s})$ is nonempty. 
\end{Def}

In this paper, we always fix $C$.
Often, we will consider $u$ also as a $G$-equivariant map from $P$ to $W$.
Two quasimaps  $(P, u)$ and $(P', u')$ are said isomorphic if there is an isomorphism
\begin{equation}\label{auto} \xymatrix{ P \ar[rr]^{\phi}  \ar[dr] & & P' \ar[dl]  \\
             & C &  }\end{equation}  of principal $G$-bundles
preserving sections, that is, $\tilde{\phi} \circ u = u'$ where 
$\tilde{\phi}: P\times _G W\ra P'\times _G W $ is the map induced from $\phi$. 

The {\em degree class} of  $(P, u)$ is defined to be the degree class of $P$.      
The {\em moduli stack} $$QG _\beta$$ of stable quasimaps with a rigid domain $C$ and a fixed degree  class $\beta$
is an open substack of Hom stack $\Hom (C, [W/G])$ which is an Artin stack locally of finite type over $\CC$ (see  \cite[Proposition 2.11]{Lieb}).

\subsection{Some results from \cite{CKM}}
The space $QG _\beta$ is a finite type DM stack by the boundedness theorem \cite[Theorem 3.2.5]{CKM}. 
In the paper \cite{CKM}, $QG_\beta$ is denoted by  $$Qmap _{g(C),0} ([W ^s/G] , \beta; C )$$ and called the 
{\em quasimap graph space}. (Here we do not need the condition that $W^s=W^{ss}$.) 

\medskip

There is a natural morphism from $QG _\beta$ to the {\em affine quotient} $\affW:=\mathrm{Spec}\/ \CC [W]^G$ obtained by assignment:
\[ (P, u) \mapsto  \mathrm{Im}(C\stackrel{[u]}{\ra} [W/G] \ra \affW), \] where 
the composite is a constant map since $C$ is projective and the affine quotient is affine.

It is shown in \cite{CKM} that $QG _\beta$ is  proper over $\affW$ if  $W^s = W^{ss}$ and $G$ acts on $W^s$ freely
(for the proof see also \S \ref{T-action}).


\subsection{Symmetric obstruction theory}\label{symsubsection}

When $C$ is an elliptic curve, we shall show that
$QG_\beta$ carries a symmetric obstruction theory.
We begin with recalling the definition.

\begin{Def} (\cite{BF})
A perfect obstruction theory $E\ra \mathbb{L}_{\mathfrak{M}}$ for a finite type DM-stack $\mathfrak{M}$ is called {\em symmetric} if 
it is endowed with an isomorphism $\sigma : E\ra E ^\vee [1]$ 
in the derived category of coherent sheaves on $\mathfrak{M}$ such that
$\sigma ^\vee[1] = \sigma $ holds, where $\mathbb{L}_{\mathfrak{M}}$ is
the cotangent complex of $\mathfrak{M}$ relative to $\CC$.
\end{Def}

Denote the natural maps as in diagram
 \begin{equation}\label{notation}\xymatrix{
 C\times QG_\beta  \ar[r]^-{[u]}  \ar[d]_{\pi} & [W/G] \\
 QG_\beta &  }. \quad\quad     \end{equation}
We use the universal maps (e.g. $u$) by the same letters.
Now we can obtain 
an obstruction theory    
\begin{equation}\label{ob1} \mathbb{E}:=(R^\bullet\pi _* [u]^*\mathbb{T}_{[W/G]} )^\vee\ra \mathbb{L}_{QG_\beta} \end{equation}
as in \cite[Proposition 6.3]{BF1}.

Let $\mathcal{P}$ denote the universal $G$-bundle on $C\times QG_\beta$
and let $$\mathcal{P}(\mathbb{T}_{[W/G]})$$ denote the complex on $\mathcal{P} \times _G W$ 
associated to the $G$-equivariant perfect complex $\mathbb{T}_{[W/G]}$. 
It is useful to note that 
\[  [u]^*\mathbb{T}_{[W/G]} = u^*\mathcal{P}(\mathbb{T}_{[W/G]}). \]

\begin{Prop}\label{sym} If $C$ is an elliptic curve, then the obstruction theory  \eqref{ob1}
for $QG _\beta$ is a symmetric obstruction theory.
\end{Prop}

\begin{proof}  
By the Grothendieck duality and (\ref{Fsym}), 
\begin{align*} (R^\bullet\pi _* ([u]^*\mathbb{T}_{[W/G]}) )^\vee 
& \cong R^\bullet\pi _* (([u]^*\mathbb{T}_{[W/G]}^\vee )\ot\omega _{\pi}[1]) \\
 & \cong R^\bullet\pi _*( [u]^*\mathbb{T}_{[W/G]}) [1]. \end{align*} 
Denote by $\sigma$ the composite of the isomorphisms. 
We need to check that $\sigma ^\vee = \sigma [-1]$.
This can be seen by the equality 
$\mathrm{tr}(i(t_{0,0})\ot t_{0,1} ) = \mathrm{tr}(i (t_{0,1})\ot t_{0,0})$  for $t_{0,0}\in H^0(C, [u]^*\mathbb{T}_{[W/G]}),
t_{0,1}\in H^1(C, [u]^*\mathbb{T}_{[W/G]})$, where $\mathrm{tr}$ is the trace map and
 $i$ is the isomorphism $[u]^*(\mathbb{T}_{[W/G]})  \ra [u]^*(\mathbb{T}_{[W/G]}) ^\vee\ot \omega _C$.

The symmetry together with the fact that there are no nontrivial infinitesimal automorphisms
implies that $\mathbb{E}$ is a two term perfect complex in the derived category of coherent sheaves. 
 \end{proof}


\section{Twisted Quasimaps}

By twisting quasimaps, we shall
generalize the examples of symmetric obstruction theory above for any smooth projective  curve $C$.
The ideas of the  twisting in our context and Theorem \ref{twisted-symmetric} are originated from \cite{Dia1, Dia}.

\subsection{Set-up} Let $A$ be a finite index set and for $a\in A$
let $(M_{a^+}, M_{a^-})$ be a pair of  line bundles on $C$ 
equipped with an isomorphism $M_{a^+}\ot  M_{a^-}\cong \omega _C$.
For $a\in A$, consider a $G$-representation space $V_{a^+}$ and 
denote by $V_{a^-}$ the $G$-space $V_{a^+}^*$, the dual space of $V_{a^+}$.
Let $X$ be the canonical symplectic space 
$$T^*V_+=V_+\oplus V_-, \text{ where } V_+ :=\bigoplus _{a\in A} V_{a^+}.$$
Suppose that the moment map $\mu$ is the sum of each moment map  $\mu _a $, {\em bilinear} on $X_a:=T^*V_{a^+}$.


\subsection{Twisted quasimaps}
Consider the total space of a $G$-equivariant vector bundle on $C$:
\[  \bM\ot X:= \mathrm{Tot} (\bigoplus _{a\in A, i=\pm} M_{a^i}\ot V_{a^i} ) \]
and define the pullback of $\mu _a$:
\[\begin{array}{cccc}
\tilde{\mu}_a: & \Tot (\sum _i M_{a^i}\ot V_{a^i} ) & \ra & \Tot (\omega _C\ot \mathfrak{g}^*), \\
         &       m_{a^+} \ot v_{a^+} + m_{a^-} \ot  v_{a^-} & \mapsto & m_{a^+}m_{a^-}  \ot \mu _{a} (v_{a^+}, v_{a^-}) , \end{array}   \]
         which is well-defined by the bilinearity of $\mu _a$.
This in turn defines 
\[ \bM\ot W :=\tilde{\mu} ^{-1} (0), \text{ where } \tilde{\mu}:= \sum _{a\in A} \tilde{\mu}_a.\]

Note that $\bM\ot W$ is a $G$-invariant subvariety of $\bM\ot X$ (i.e., $G\cdot \bM\ot W \subset \bM\ot W$).
We denote by $\tilde{\alpha}$ the induced homomorphism $G\ra \mathrm{Aut}(\bM\ot V)$.
Locally over $C$, $\bM \ot X$ (resp. $\bM \ot W$) can be $G$-equivariantly trivialized with fiber $X$ (resp. $W$).

Let $P$ be a principal $G$-bundle and write
\[ P_\bM (X) := (P\times _C (\bM\ot X)) / G , \] which is  (the total space of) a vector bundle on
$C$. Let $u$ be a section of $P_{\bM}(X)$.

\begin{Def} 
A pair $(P, u)$ is called  a {\em $\bM$-twisted quasimap} to $\XXX$
if it is a map over $C$ from $C$ to the stack quotient $[\bM\ot W/G]$.
\end{Def}

Let $\Gamma (C, [\bM\ot W/G])$ be the stack of $\bM$-twisted quasimaps $(P, u)$: 
\[ \xymatrix{ C  \ar[r]^{(P,u)\ \ \ \ \ \ \ }\ar[rd]_{=}  & [\bM\ot W/ G]  \ar[d]^{\rho} \\
                                  & C },  \text{ with } \rho \circ (P, u) = \mathrm{id} \] for 
                                  the canonical projection $\rho$.
The stack is a closed substack of the Artin stack $\Hom (C, [\bM\ot W/G])$.

With respect to the character $\chi$, we may consider the stable locus $(\bM\ot W)^s$ and 
its quotient stack $ [(\bM\ot W)^s / G ] $.

\begin{Def} We say that a $\bM$-twisted quasimap $[u]:=(P, u)$ to $\XXX$ is {\em stable} if 
$[u]^{-1} ( [(\bM\ot W)^s / G ] )$ is nonempty.
\end{Def}

Define $$QG _{\bM, \beta}$$ to be the stack of stable $\bM$-twisted quasimaps with the degree class $\beta$.

\begin{Thm}\label{twisted-symmetric}
\begin{enumerate}
\item The stack $QG_{\bM, \beta}$  is a finite type DM stack over $\CC$ with a canonical symmetric perfect obstruction theory.

\item 
Assume that $W^s = W^{ss}$ and $G$ acts on $W^s$ freely.
Suppose that  for all $a^i$ there is a torus $T$ action on $V_{a^i}$ which is linear and commutes with the $G$ action such that  the induced 
$T$ actions on $X$ and $\bM\ot X$ preserve subvarieties $W$ and $\bM\ot W$.
If $(\affW )^T$ is isolated,
then the $T$-fixed part $QG_{\bM, \beta} ^T$ of $QG_{\bM , \beta}$ is proper over $\CC$. 

\end{enumerate}

\end{Thm}

\begin{proof}

1) Since a stable $\bM$-twisted quasimap is a rational map to a DM stack $[(\bM\ot W)^s/G]$, it has no infinitesimal automorphisms. 
 Thus the Artin stack $QG_{\bM, \beta}$ is DM. It is of finite type over $\CC$ by \cite[Theorem 3.2.5]{CKM}.
The obstruction part will be proven in \S \ref{SymmetricObs}. 

2) This will be proven in \S \ref{T-action}.
\end{proof}

 
\subsection{Obstruction Theory}\label{SymmetricObs}
Denote universal morphisms as in the following commuting diagram:
 \[ \xymatrix{ C\times QG _{\bM, \beta} \ar[r]^{[u]\ \ } \ar[rd]_{\pi _C}   \ar[d]_{\pi}  & [\bM\ot W / G ] \ar[d]^{\rho}  \\
                                       QG _{\bM, \beta}           &         C .} \]

 From this, immediately we obtain a natural obstruction theory for $QG _{\bM, \beta}$:
\begin{equation}\label{twisted-obstruction} \mathbb{E}_\bM :=(R\pi _* [u]^*\mathbb{T}_\rho)^\vee
\ra \mathbb{L} _{QG _{\bM, \beta}}. \end{equation}
Here the relative tangent complex $\mathbb{T}_\rho$ of $\rho$ is explicitly 
\[ \begin{array}{ccccc}
    \mathcal{O}_{\bM\ot W} \ot   \mathfrak{g}  & \stackrel{d\tilde{\alpha}} {\longrightarrow} &   \mathcal{T}_{\bM\ot X}|_{\bM\ot W}  & \stackrel{d\tilde{\mu}}{\longrightarrow} & 
      \rho ^*\omega _C\ot \mathfrak{g}^*\\
      \\
     \begin{array}{c} g \\ \text{at } \sum m_{a^i}\ot w_{a^i} \end{array} &\mapsto & \sum \begin{array}{c} m_{a^i}    \\ \ot d\alpha (g)_{w_{a^i}}\end{array} & & \\
      \\
                                &    &                                     \sum  m_{a^i} \ot v_{a^i} & \mapsto &  \sum  \begin{array}{c}  m_{a^+}m_{a^-} \\ 
                                                                                                                 \ot d\mu _a (v_{a^+}, v_{a^-}) \end{array}
\end{array} \]       
concentrated at $[-1, 1]$.
                                          
Parallel to \eqref{Fsym},
\begin{equation}\label{Tsym} \mathbb{T}_\rho ^\vee \ot \rho ^* \omega _{C} \cong \mathbb{T}_\rho  \end{equation}
by the symplectic form and the isomorphisms $M_{a^+}\ot M_{a^-} \cong \omega _C$.
Therefore, \eqref{twisted-obstruction}  is  a symmetric
obstruction theory for $QG _{\bM, \beta}$ as in the proof of Proposition \ref{sym}.

\subsection{$T$-action}\label{T-action}

We verify Theorem \ref{twisted-symmetric} part (2) by the valuative criterion for properness, following \cite{CKM}.
First we note that by the assumption $(\affW)^T =$ isolated, $(\bM\ot W/\!\!/_{\! 0} G )^T$ is isomorphic to a projective scheme $(\affW )^T\times C$;
and by the assumption $W^s = W^{ss}$ with the free $G$-action, 
$[(\bM\ot W)^s/G]^T = [(\bM\ot W)^{ss}/G] ^T = (\bM\ot W /\!\!/G)^T$ which is a projective scheme over  $(\bM\ot W/\!\!/_{\! 0} G )^T$.

Let $(\Delta, 0)$ be a pointed smooth curve and consider a family of $T$-fixed stable quasimaps $(P, u): C\times (\Delta \setminus\{0\}) \ra [\bM\ot W/G]$.
Since $[(\bM\ot W)^s/G]^T$ is a projective scheme, the stable quasimap $(P, u)$ extends uniquely to
a stable quasimap $(\bar{P}, \bar{u}): C\times \Delta \ra  [\bM\ot W/G]$ possibly except finitely many points on the central fiber of $C\times \Delta$.
We may extend $(\bar{P}, \bar{u})$ to a quasimap defined everywhere on $C\times \Delta$ by extending the principal bundle $P$ and the section $u$.
The latter extensions are  possible due to \cite[Lemma 4.3.2]{CKM} and Hartogs' theorem, respectively. The uniqueness of these extensions
are clear.

\begin{Rmk}  If  the isomorphism \eqref{Tsym} is $T$-equivariant, then
the $T$-fixed part of $\mathbb{E}_\bM$ (see \cite[Proposition 1]{GP}) is 
a symmetric obstruction theory for $QG_{\bM, \beta} ^T$.
\end{Rmk}


\section{The Quiver Example}

\subsection{Nakajima's quiver varieties} 
The typical examples of $\XXX$, to which Theorem \ref{twisted-symmetric} can be applied,
are Nakajima's quiver varieties.
We set up quiver varieties, basically following \cite{Gin, Na}.

Let $Q$ be a finite quiver, which means that it is equipped with two
finite sets $Q_0$, $Q_1$ and two maps $t, h: Q_1\ra Q_0$.
We call $Q_0$ (resp. $Q_1$) the vertex (resp. arrow)  set of $Q$ and 
$ta$ (resp. $ha$) the tail (resp. head) of arrow $a\in Q_1$.
Let $\overline{Q}$ be the double quiver of  $Q$. 
It is defined as follows. 
The vertex set $\overline{Q} _0$ of $\overline{Q}$ is
exactly $Q_0$.  For each arrow $a$ in $Q_1$, 
create exactly two associated arrows $a^+$, $a^{-}$ of $\overline{Q}$, by
making the head (resp. tail) of $a^+$ (resp. $a^-$) = the head (resp. tail) of $a$.

Given a {\em dimension vector} $v=(v_i)\in \NN^{ Q_0}$, let 
\[ \mathrm{Rep}(\overline{Q}, v)
:=  \bigoplus _{a\in Q_1} \left( \Hom (\CC ^{v_{ta^+}}, \CC ^{v_{ha^+}})\oplus \Hom (\CC ^{v_{ta^-}}, \CC ^{v_{ha^-}})\right). \]
After $\mathrm{Rep}(\overline{Q}, v) $ being canonically identified with the total space of 
the cotangent bundle of  $$\bigoplus _{a\in Q_1} \Hom (\CC ^{v_{ta^+}}, \CC ^{v_{ha^+}}), $$
the space 
\begin{equation}\label{Xrep} X:=\mathrm{Rep}(\overline{Q}, v) \end{equation} can be regarded as a holomorphic symplectic manifold 
with the canonical holomorphic symplectic form $\omega$. The linear symplectic form is defined by
\[ \omega (A,B) = \sum _a \tr (A_{a^+}B_{a^-}) -  \tr (A_{a^-}B_{a^+}) \] for tangent vectors 
$A, B \in  T_{x} \mathrm{Rep}(\overline{Q}, v)=\mathrm{Rep}(\overline{Q}, v)$ at $x\in X$. 

Fix a subset $Q_0'$ of $Q_0$ and let $Q_0'':=Q_0\setminus Q_0'$.
Let $$G:=\prod _{i\in Q_0'} GL_{v_i}(\CC ).$$ 
There is a natural hamiltonian $G$-action on $X$ with a moment map
\begin{equation}\label{Wmu}  \mu (x) =  \bigoplus _{i\in Q_0'} (\sum _{a\in \overline{Q}_1: t\overline{a}=i} (-1)^{|a|}x_a x_{\overline{a}})   \text{ for } 
x=(x_a)_{a\in \overline{Q}_1} \in X \end{equation}
 where $\overline{a}$ denotes the opposite arrow of $a$
 and $|b^+|=0$, $|b^-|=1$ for $b\in Q_1$. Here we identified $\mathfrak{g}$ with its dual by trace.
 The equation (\ref{moment}) can be checked easily since for $g\in \mathfrak{g}$,
 $d\ka (g)$  is the linear vector field $(g_{ha}x_a - x_a g_{ta})_{a} \in \mathrm{Rep}(\overline{Q}, v)$.

Let $$\lambda = \sum _{i\in Q_0'} \lambda _i \mathrm{Id}_{\mathrm{End} _{v_i}(\CC )} \in
 \bigoplus _{i\in Q_0'} \mathrm{End} _{v_i}(\CC ), \ \lambda _i\in \CC$$ and
choose a  character $\chi = (\theta _i) \in\ZZ ^{Q_0'}$ of $G$ by sending $g\in G$ to 
$\prod _{i\in Q_0'} (\det g_i)^{\theta _i}\in\CC ^{\times}$. This defines a 
linearization $L=W \times \CC _{\chi}$ where $W:= \mu ^{-1}(\lambda )$.
Now we may consider the holomorphic symplectic quotient $\XXX:=W/\!\!/_{\chi}G$. We call the quotient a {\em quiver variety}.


\subsection{Twisted quasimap to a Nakajima quiver variety}
In this subsection let $\lambda =0$.
For each arrow $a\in Q_1$, fix two line bundles $M_{a^{\pm}}$ on $C$ with a fixed isomorphism
$M_{a^+}\ot M_{a^-}\ra \omega _C$. Below we identify $M_{a^+}\ot M_{a^-}\cong \omega _C$ 
and for $M_{a^-}\ot M_{a^+}\cong \omega _C$ we will use 
the natural isomorphism followed by the given one: $M_{a^-}\ot M_{a^+}\ra M_{a^+}\ot M_{a^-} \ra \omega _C$.

We rewrite the notion of $\bM$-twisted quasimaps to the quiver variety $\XXX$ in terms of twisted quiver bundles. 
\begin{Def} 
A pair $(P, u)$ is called a {\em $\bf M$-twisted quasimap} to $X/\!\!/\!\!/G$ with a rigid domain curve $C$ if:
\begin{itemize}
\item $P$ is a principal $G$-bundle on $C$.
\item  $u=(u_{a})_{a\in \overline{Q}_1}$ where 
\[ u_a \in \Gamma (C, P\times _G \Hom (\CC ^{v_{ta}}, \CC^{v_{ha}}  )\ot M_{a}). \]
\item The section $u$ satisfies the \lq\lq moment map" equation (for $\kl=0$): for all $i\in Q_0'$,
\[  \sum _{a\in \overline{Q}_1: t\overline{a}=i}(-1)^{|a|} (u_a\ot \mathrm{Id}_{M_{\overline{a}}})\circ u_{\overline{a}})
= 0.\] 
\end{itemize}

An $\bf M$-twisted quasimap is called {\em stable} 
if $u$ hits unstable locus only at finitely many points of $C$.
\end{Def}

This $\bM$-twisted quasimap is  
a generalization of a stable ADHM sheaf of Diaconescu \cite{Dia1} from the point of view of 
quasimaps \cite{CK, CKM}.

\subsection{Obstruction theory}

We note that $[u]^*\mathbb{T}_\rho$ becomes
\begin{eqnarray*}    0  \ra  \mathcal{P}\times _G \mathfrak{g}
&\ra & \begin{array}{r} \bigoplus _{a\in Q_1}\left(
 \mathcal{H}om (\cV_{ta^+}, \cV _{ha^+})\ot \pi _C^*M_{a^+}  \right. \\
\left. \oplus \mathcal{H}om (\cV _{ta^-}, \cV _{ha^-}) \ot \pi _C ^* M_{a^-} \right)\end{array} \\ 
  & \ra &  ( \mathcal{P}\times _G \mathfrak{g}^*)  \ot \pi _C^* \omega _{C} \ra 0, \nonumber \end{eqnarray*} 
 where $\cV _i = \mathcal{P}\times _G \CC ^{v_i}$. 
This is a generalization of the complex (4.3) in \cite[Definition 4.3]{Dia1}.


\section{Stabilities on Quiver bundles}\label{Stab}

\subsection{King's stability}
The stability with respect to  $\chi$ we used in the previous section can be rephrased
as a Rudakov's stability condition on a suitable abelian category
of representations of the path algebra $\CC\overline{Q}$ with relations (see \cite{Ru}, \cite[\S 2.3]{Gin}). 

The path algebra is a $\CC$-algebra spanned by, 
as a $\CC$-vector space, all
finite paths $a_n...a_1$ of consecutive arrows and an extra arrow $e_i$, for each $i\in Q_0$, where $a_l\in \overline{Q}_1$
and $ha_l=ta_{l+1}$ for all $l$. The product is given by a sort of compositions. Namely,
$(a_{j_m}...a_{j_1}) \cdot (a_{k_n} .... a_{k_1}) $ is $a_{j_n}...a_{j_1} a_{k_n} .... a_{k_1} $ if $ha_{k_n}=ta_{j_1}$,
$0$ otherwise.  The generators are subject to relations:
$e_i^2 = e_i$; $e_i a$ is $a$ if $ha=i$, $0$ otherwise; and $ae_i $ is $a$ if $ta=i$, $0$ otherwise.
 Impose one more relation coming from the moment map together with an element $\lambda \in \CC ^{Q_0'}$:
\[ \sum _{a\in \overline{Q}_1 : t\overline{a}\in Q_1'} (-1)^{|a|} a\overline{a} = \sum _{i\in Q_0'} \lambda _i e_i \] 
which will be denoted symbolically by $\mu - \lambda =0$.  

Denote by $(\mu -\kl )$ the two-sided ideal generated by $\mu - \kl$.
Note that a $\CC \overline{Q}/(\mu -\kl )$-module $V$ amounts to data $(V_i, \phi _a) _{i\in Q_0, a\in \overline{Q}_1}$ 
where $V_i$ is a $\CC$-vector space and
$\phi _a$ is a homomorphism from $V_{ta}$  to $V_{ha}$ subject to the condition coming from $\mu - \lambda =0$.
A homomorphism from  a module $(V_i, \phi _a) $ to another  $(V_i', \phi _a') $ is nothing but a collection $(\varphi _i)_{i\in Q_0}$
of linear maps $\varphi _i: V_i \ra V_i'$ 
making $\phi _a' \circ \varphi _{ta} = \varphi _{ha}\circ \phi _a $ for every $a\in \overline{Q}_1$.

Now, following \cite{King}, we are ready to reformulate the $\chi$-stability of $V=(V_i, \phi _a) \in \mathrm{Rep}(\overline{Q}, v)$ 
satisfying the moment map relation $\mu -\lambda =0$. 
Suppose $v_0:= \sum _{i\in Q_0''}  \dim  V_i\ne 0$.
Let $\theta  _0= -(\sum _{i\in Q_0'} \theta _i v_i)/ v_0$ and for a $\CC \overline{Q}/(\mu -\lambda)$-module $S$ write
\[\theta (S) = \theta _0 \dim S_0 + \sum _{i\in Q_0'}\theta _i \dim S_i, \]
where $S_0 := \oplus _{i\in Q_0''} \dim S_i$.
Note that $\theta (V)=0$.

\begin{Thm}\label{King}
The followings are equivalent.
\begin{enumerate} 
\item $V$ is semistable (resp. stable) with respect to $\chi=(\theta _i)$ as a point in $\mu ^{-1}(\lambda) \subset \mathrm{Rep}(\overline{Q}, v)$.
\item $\theta (S) \ge 0$  (resp. $\theta (S) > 0$) for every nonzero, proper, submodule $S$ of $V$ 
with $S_0=V_0$ or $S_0=0$.
\end{enumerate}
\end{Thm}
\begin{proof}
A slightly modified argument of \cite[\S 3]{King}, which uses Proposition \ref{NC}, works with group $G$, too. 
We provide its details.
For a 1-parameter subgroup $\kappa :\CC ^\times \ra G$, we may consider a linear action on the 
vector space $V = \oplus _{i\in Q_0} \CC ^{v_i} $ via the standard action of $G=\prod _{i\in Q_0'} GL_{v_i}(\CC )$.
If $V^{n}:= \{ p\in V: \kappa (t)\cdot p = t^m p, m\ge n,  \forall t\in \CC ^\times \}$, we obtain
a $\ZZ$-filtration $\cdots \supset V^{n-1} \supset V^{n} \supset \cdots$. 

For (2) $\Rightarrow$ (1), 
suppose $\lim _{t\ra 0} \kappa (t)\cdot (V_i, \phi _a)$ exists for a nontrivial $\kappa$. 
The existence means that for every $n$, $V^n$ is a submodule of $V$. Since $(V^n)_0=V_0$ or $0$, we see 
that $\theta (V^n)\ge 0$.
Since the $\kappa$-action is nontrivial, 
some $V^n$ is a proper, nonzero submodule of $V$. Now the proof follows from the identity
$\langle \chi , \kappa \rangle = 
\sum _{n\in \ZZ} n \theta (V^n/V^{n+1}) = \sum_{n\in \ZZ}  \theta (V^n)$.

For (1) $\Rightarrow$ (2), let $S$ be a nonzero, proper submodule of $V$.
If $S_0= V_0$, then there is
a nontrivial $\kappa$ such that $V^{-1}=V$, $V^{0} = S$ and $V^1=0$.
If $S_0=\{ 0\}$, then  there is
a nontrivial $\kappa$ such that $V^{0} = V$, $V^1=S$ and $V^{2}=\{ 0\}$.
In either case, $\lim _{t\ra 0} \kappa (t)\cdot (V_i, \phi _a)$ exists 
and $\langle \chi , \kappa \rangle = 
\sum _{n\in \ZZ} n \theta (V^n/V^{n+1}) = \sum _{n\in\ZZ} \theta (V^n) = \theta (S)$. Therefore, we conclude the proof by Proposition \ref{NC}. 
\end{proof}

\medskip

This motivates the following. First, {\em from now on we assume that
$Q_0''$ has exactly one vertex $0$. } Fix  a finite dimensional vector space $K$.

\medskip

\begin{Def} 
Denote by  $\mathrm{Rep}_{\lambda}(\overline{Q}, K)$ the category  whose objects
are finite dimensional $\CC \overline{Q}/(\mu -\kl )$-modules $V$ with an identification
$V_0=\bar{V_0}\ot _{\CC} K$ for some vector space $\bar{V_0}$
and whose morphisms, say from $V$ to $V'$ 
are module homomorphisms $(\varphi _i)_{i\in Q_0}$
satisfying the condition:
 $\varphi _0 = \bar{\varphi _0}\ot \mathrm{id}_K : V_0=\bar{V_0}\ot K \ra \bar{V_0}'\ot K = V_0'$ where 
 $\bar{\varphi _0}$ is a $\CC$-linear map from $\bar{V_0}$ to $\bar{V_0}'$.
 It is clear that 
$\mathrm{Rep}_{\lambda}(\overline{Q}, K)$  is an abelian category.
\end{Def}

Define a homomorphism 
\[ Z :\ZZ ^2\ra \CC,  \ \ 
                                                                        (x, y)  \mapsto y +\sqrt{-1} x. \]
For $V\in \mathrm{Rep}_{\lambda} (\overline{Q}, K)$,
let  $Z(V)=Z(v_0,v_1) $ where $v_0=\dim V_0$, $v_1=\dim V_1$ and $V_1:= \oplus _{i\in Q_0'} V_i$.
Note that for nonzero $V$, $Z(V)\neq 0$  and $0 \le \mathrm{Arg}(Z(V)) \le \pi /2$. 

For a $V\in\mathrm{Rep}_\kl (\overline{Q}, K)$ with $v_0\ne 0$, if
we take  $$\theta _0:= - v_1/v_0,\  \theta _i := 1, \forall i\in Q_0'$$ 
then
$\mathrm{Arg}Z (S)  \le \mathrm{Arg} Z (V) $ if and only if $\theta (S) \ge 0$
(resp. $\mathrm{Arg}Z (S)  <  \mathrm{Arg} Z (V) $  if and only if $\theta (S) > 0$)
for every nonzero proper subobject $S$ of $V$ in $\mathrm{Rep}_\kl (\overline{Q}, K)$. 
Therefore, by Theorem \ref{King}, Rudakov's stability (\cite{Ru})
defined via the stability function $Z$ coincides with King's $\theta$-stability.


\subsection{Quiver sheaves} 
In this subsection, for every $i\in Q_0'$ fix $\lambda _i \in \Gamma (C, \omega _C)$.                                                                    
For a systematic study of stabilities on quasimaps we interpret quasimaps as linear objects. 
Fix  a finite dimensional vector space $K$ as before.

\begin{Def} A 
data $(E_i, \phi _a)_{i\in Q_1, a\in \overline{Q}_1}$ 
is called {\em $\bf M$-twisted quiver (coherent) sheaf on $C$} with respect to $(\overline{Q}, \lambda)$ 
 if $E_i$ is a (coherent) sheaf of $\cO _C$-modules;  $E_0$ is 
$\bar{V_0}\ot _{\CC}(K\ot _{\CC} \mathcal{O}_C)$ for some finite dimensional vector space $\bar{V_0}$; and $\phi _a$ is
a $\mathcal{O}_C$-module homomorphism from $E_{ta} $ to $E_{ha}\ot M_a$. 
The homomorphisms are subject to 
relation (which will be denoted also by $\mu  -\lambda =0$):
\begin{equation}\label{relations} \sum _{a\in \overline{Q}_1 : t\overline{a}=i} (-1)^{|a|} 
(\phi _a\otimes \mathrm{Id}_{M_{\overline{a}}})\circ \phi _{\overline{a}} -\mathrm{Id}_{E_i}\ot \lambda _i= 0, \ \ \forall i\in Q_0'\end{equation}
 unless stated otherwise.
\end{Def}

\begin{Rmk} 
For the history of the studies of (the moduli spaces of) twisted quiver sheaves usually without the relation, 
see \cite{A, AG, CH, GK, Sh} and references therein. See also \cite{Sz1}.
\end{Rmk}

Like a quiver representation and a $\mathcal{O}_C$-sheaf, a quiver sheaf can be considered as
a module of the $\bf M$-twisted path algebra ${\bf M}\overline{Q}$ over $\mathcal{O}_C$ (\cite{GK, AG}). 
For each path $p=a_m...a_1$, let
 $M^{\vee}_p=M_{a_m}^\vee \ot _{\mathcal{O}_C} M_{a_{m-1}}^\vee \ot _{\mathcal{O}_C} ...\ot _{\mathcal{O}_C}M_{a_1}^\vee$
 and for $e_i$, let $M_{e_i}^{\vee} = \mathcal{O}_C$.
 Let $${\bf M}\overline{Q}/(\mu -\lambda )=(\bigoplus _{\text{all paths } p} M^{\vee}_p) /(\mu -\lambda )$$
 which has a $\mathcal{O}_C$-algebra structure similar to the path algebra $\CC \overline{Q}$. Here $(\mu -\lambda )$ is
 the two-sided ideal generated by the relations (\ref{relations}) for \lq\lq abstract" $\phi _a$: For every local 
 section $\xi\in \omega _C ^\vee$, consider
 $$(\mu -\lambda )_i(\xi ) :=\sum _{ha=i} (-1)^{|a|} \xi _a \ot \xi _{\bar{a}} - \langle \xi, \lambda _i \rangle e_i$$ where
 $\xi _a\ot \xi _{\bar{a}}$ is an element in $M_a^\vee\ot M_{\bar{a}}^\vee$ corresponding to $\xi$ 
 under the given isomorphism $M_a^\vee\ot M_{\bar{a}}^\vee\cong \omega _C^\vee$.
 The ideal sheaf $\mu - \lambda$ is defined to be the ideal sheaf generated by $(\mu -\lambda )_i(\xi )$ for all  $i\in Q_0'$, $\xi\in \omega _C^\vee$.
 
 Given a ${\bf M}\overline{Q}/(\mu -\lambda )$-module structure on $E$, a $\bf M$-twisted quiver sheaf can be
 associated by letting $E_i=M_{e_i}^{\vee}E$ and $(\phi _a )_{|_U} (m_a\ot s )= m_a s $ for an open set $U\subset C$, 
 $m_a\in M_a^\vee(U)$, $s\in E_{ta}(U)$. Here we regard $\phi _a$ as a homomorphism from  $M_{a}^\vee\ot E_{ta}$ to  $E_{ha}$.
 Conversely, a quiver sheaf defines  a module structure on $\oplus E_i$. For full details, see \cite[Lemma 2.1, \S 4]{CH}, \cite[Proposition 2.3]{KL}, 
 \cite[Proposition 5.1]{AG}.

\medskip

Denote by $\Rep$ the abelian category of $\bf M$-twisted quiver coherent sheaves $E$ with framing
$E_0=\bar{V_0}\ot _\CC  (K\ot _\CC \mathcal{O}_C)$ for some finite dimensional vector space $\bar{V_0}$.
A morphism from $(E_i, \phi _a)$ to $(E_i', \phi _a')$ is, by definition, a collection $(\varphi _i)_{i\in Q_0}$
of $\mathcal{O}_C$-homomorphism $\varphi _i: E_i \ra E_i'$ 
making $\phi _a' \circ \varphi _{ta} = (\varphi _{ha}\ot 1_{M_a})\circ \phi _a $ for every $a\in \overline{Q}_1$
and satisfying $\varphi _0 = \bar{\varphi _0} \ot \mathrm{id}_{K\ot _\CC \mathcal{O}_C}$ for some $\CC$-linear map $\bar{\varphi _0}$.
Denoted by $$\Repkl$$ to be the full subcategory of $\Rep$ whose objects
satisfy the moment map relation \eqref{relations}.

For $\delta >0$, define a homomorphism 
$Z_{\delta} :  \ZZ ^3 \ra \CC$
 by assignments
\[ \begin{array}{rcl} Z(v_0, v_1, d) & = & Z _{(1)} (v_1,d) + Z_{(2)} (v_0, v_1), \\ 
Z _{(1)}(v_1,d) &=&  \frac{v_1}{2} +\sqrt{-1}d, \\ 
Z _{(2)} (v_0, v_1) &=& \frac{v_1}{2}+   \sqrt{-1}\delta v_0. \end{array}\] 
 Also define $Z_{\delta}(E) \in\CC$ by the rank-degree map
$\Rep \ra  \ZZ^3$, $E\mapsto (\rank E_0, \rank E_1, \deg E_1)$ followed by $Z$,
where $E_1:=\oplus _{i\in Q_0'} E_i$.

We remark that the homomorphism (followed by $\pi/2$-rotation) 
is a stability function on $\ZZ ^3$
with Harder-Narasimhan property in the sense of Bridgeland (see \cite[Proposition 2.4]{Br}). 
In more elementary way, this gives rise to a Rudakov's stability \cite{Ru} in an abelian category.

Let $\mu _{\delta} (E)\in (-\infty , \infty ]$ be the slope of $Z_\delta (E)$ for a nonzero quiver sheaf $E$. 
We abuse notation by letting $\mu$  stand for slopes as well as moment maps.
This shouldn't cause any confusion.

Following the $\delta$-stability introduced in \cite[Definition 2.1]{Dia} for ADHM case, we come to this. 
\begin{Def}
An  $\bf M$-twisted quiver coherent sheaf $$E \in \Repkl, \text{ with } \rank E_0\neq 0$$
is called {\em $\delta$-semistable} (resp. {\em $\delta$-stable}) if $\mu _{\delta} (E')\le \mu _{\delta} (E)$
(resp. $\mu _{\delta} (E') < \mu _{\delta} (E)$)
for every nonzero proper subobject $E'$ of $E$ in $\Repkl$.
\end{Def}
It is necessary that a $\delta$-semistable quiver sheaf $E$ is locally free
(i.e., a quiver sheaf with $E_i$ 
being locally free sheaf for every $i$). We call a locally free quiver sheaf a {\em quiver bundle}.

\bigskip

{\em From now on assume that
$\kl =0$ and by a $\bM$-twisted quiver sheaf $E$ we mean $E\in \Repklo$ so that 
$E$ satisfies the relation \eqref{relations}.}

\bigskip
 If $X:=\mathrm{Rep} (\overline{Q}, v)$ as in \eqref{Xrep},
a $\bf M$-twisted quasimap to $X/\!\!/\!\!/G$ with degree $\beta$
amounts to a $\bf M$-twisted quiver bundle  with 
\[ \mathrm{rank} E_i = v_i,\  \deg E_i = \beta (\mathrm{det}_i),\]  where
$\det _i$ is the character of $G$ given by the determinant of $i$-th general linear group. 
More precisely speaking, if $W:=\mu ^{-1}(0)$ (see \eqref{Wmu}), there is a natural 1-morphism from $\Gamma (C,  [\bM\ot W/G] )$ to
 the moduli stack of $\bf M$-twisted quiver bundles on $C$ with $E_i=P\times _G \CC ^{v_i}$.
 
 Let $\mathfrak{M}_{\delta, {\bf M} , \{ v_i \} }$ be the stack of $\delta$-stable objects in $\Repklo $
 with ranks $v_i$
 and let $QG_{\delta, \bf M}$ be its inverse image under the above natural 1-morphism.

 
\begin{Thm}\label{bd} 
The moduli space $QG _{\delta, \bM, \beta}$ 
of $\delta$-stable $\bf M$-twisted quiver bundles of degree $\beta$ on $C$ 
is an algebraic space of finite type over $\CC$, equipped with
 a natural symmetric obstruction theory.
\end{Thm}
\begin{proof}
There are no nontrivial automorphisms of stable objects except overall multiplications.
The nontrivial overall multiplications of stable objects are not allowed as objects in
$QG _{\delta, \bM}$ because the framing with the trivial $G$-action and the $\delta$-stability 
remove such automorphims in \eqref{auto}. 
Thus $QG_{\delta, \bf M} \times B\CC ^* \cong  \mathfrak{M}_{\delta, \bM , \{ v_i \}}$ as stacks.
The section \S \ref{SymmetricObs} shows the complex corresponding to \eqref{twisted-obstruction}  is a
symmetric obstruction theory for
$QG _{\delta, \bM, \beta}$ since the $\delta$-stability  is an open condition as usual.

We prove the boundedness 
using Harder-Narasimhan filtration with respect to the standard slope $\mu _{st}$.
Let $(E,\phi ^a)$ be {\em $\mu _{\delta}$-semistable} quiver sheaf.
Considering $E$ as a $\mathcal{O}_C$-module $\oplus _i E_i$ on $C$, take the Harder-Narasimhan (HN) filtration
\[ 0=E^0\subset E^1 \subset ...\subset E^l =E \]
of $E$ for $\mu _{st}$ in the category of $\mathcal{O}_C$-modules.
Since the HN filtration of a direct sum is a certain sum of each HN filtration, 
we see $E^i=\oplus _{j\in Q_0} E^i_j$ for  $E^i_j:=E^i\cap E_j$.
Also, note that $E_0 = E^i_0/E^{i-1}_0$ for some $i$.

\medskip

Claim: The following inequality holds: \begin{align*} \mu _{st} (E^i/E^{i-1}) & \le N_i(\mu _\delta (E),l, \deg {\bf M}) \\
 & := \mathrm{max}\{ 0, \mu _\delta (E) \} + (l-i) \mathrm{max}\{ 0 , \deg M_a \  | \ a \in \overline{Q}_1\}   . \end{align*}

\noindent{\em Proof of Claim.}
We prove Claim by induction on $l-i$.
When $i=l$, the claim is true since $\mu _{st}(E^l/E^{l-1}) \le \mu _{st} (E^l)$.
We define a composite $$\psi ^a _i : E^i \ra E^i_{ta} \ra E_{ha}\ot M_a \ra E_{ha}\ot M_a / E_{ha}^i\ot M_a \ra E\ot M_a/ E^i\ot M_a , $$ where
the first map is the canonical epimorphism and the second map is the restriction of $\phi ^a$ to $E^i_{ta}$,
 the third map is the projection, and the last map is the canonical monomorphism.
Define $$\psi _i = \oplus _{a\in \overline{Q}_1} \psi _i^a: \oplus _a E^i \ra \oplus _a (E \ot M_a /E^i \ot M_a) , $$
where $\oplus _a E^i$ is the sum of $\overline{Q}_1$ - many copies of $E^i$.

If $\psi _i =0$, then $E^i$ is a subobject of $E$ in  $\Repklo$ 
since $E^j/E^{j-1}=E_0$ for some $j$. Hence,
$$\mu _{st}' (E^i ) := \frac{\deg E^i}{\rank E^i_1} \le \mu _{\delta} (E^i ) \le \mu _{\delta} (E)$$ which, combined with
 $$\mu _{st}(E^i/E^{i-1}) \le
\mu _{st}(E^i )=\mu _{st}'(E^i ) - \frac{\rank E^i_0\deg E^i}{\rank E^i\rank E^i_1} , $$ implies 
\begin{equation*}
 \mu _{st}(E^i /E^{i-1}) \le \begin{cases}  0 & \text{ if } 
 \deg E^i \le 0 \\
    \mu _\delta (E) & \text{ if } \deg E^i > 0 . \end{cases} \end{equation*}
Thus, $\mu _{st}(E^i/E^{i-1}) \le \mathrm{max}\{ 0, \mu _\delta (E) \} $.

If $\psi _i \ne 0$, then $\mu _{st} (E^{i'}/E^{i'-1}) \le \mu _{st} (E^{i''}/E^{i''-1} \ot M_a) $ for some $a$ and 
$i'\le i\le i''-1$. By induction hypothesis, 
$$\mu _{st} (E^{i'}/E^{i'-1}) \le N_i(\mu _\delta (E),l, \deg {\bf M}) .$$     $\Box$ 
\medskip

Now we are ready to prove the boundedness. 
Let $F$ be a nonzero subsheaf of $E$. If $$0=F^0\subset F^1 \subset  ... \subset F^m=F$$ is the
HN filtration of $F$ for $\mu _{st}$, for every $1\le k\le m$ the natural map $$F^k/F^{k-1} \ra E^i/E^{i-1}$$ is nonzero for some $i$.
Hence, $$\mu _{st} (F^k/F^{k-1}) \le \mu _{st} (E^i/E^{i-1})$$ for some $i$, which implies, combined with Claim, that
$$\mu _{st}(F)\le N_1(\mu _\delta (E), \rank (E), \deg {\bf M}).$$
Now the boundedness follows from \cite[Theorem 1.1]{Sim}.
\end{proof}

\begin{Rmk}
Note that the above proof of the boundedness holds also for $\delta$-{\em semistable} $\bf M$-twisted quiver bundles of degree $\beta$
on $C$.
\end{Rmk}

The following Lemma will be used in the proof of 
Proposition \ref{asym} below.

\begin{Lemma}\label{bbd}  Fix the curve $C$ and a branched covering map $\phi : C\ra \PP ^1$. 
Let $E$ be a vector bundle on $C$. 
\begin{enumerate}
\item If $H^1(C, E\ot \phi ^*\mathcal{O}(m_0))=0$ for some $m_0\ge 0$, then there is a number $n_0$ depending only on
$\deg E$, $\rank E$, and $m_0$ satisfying $H^1(C, E^\vee\ot \phi ^*\mathcal{O}(m))=0$ for all $m\ge n_0$.

\item If $H^1(C, E^\vee \ot L)=0$ for a line  bundle $L$ on $C$, then
$\deg F\le (|\deg L|   + |1-g|)\rank E$ whenever $F$ is a subsheaf of $E$. 
\end{enumerate}
\end{Lemma}

\begin{proof} Let $b$ be the degree of the covering map and let $r=\rank (E)$.

For (1): Since $0=H^1(C, E\ot \phi ^*\mathcal{O}(m_0))=H^1(\PP ^1, (\phi _*E) (m_0))$, we have $\phi _*E=\oplus \cO (a_i^E)$ with
$a_i^E + m_0 \ge -1$. Hence $$\sum _i a_i ^E \ge a_j^E - (m_0+1)(br-1), \text{ for any } j. $$
On the other hand, by Riemann-Roch theorem, $$\deg \phi _*E + br = \deg E + r(1-g) .$$ Therefore, 
$a_j^E \le r(1-g) +m_0br + \deg E$. If we set  $$n_0 =| r(1-g) +m_0br + \deg E| +l+1$$ for some positive $l$ with a nonzero homomorphism
$\omega _C \ra \phi ^*\mathcal{O}(l)$, the proof follows by the Serre duality. 

For (2): Note that $$H^0(C, F\ot \omega _C \ot L^\vee) \subset 
H^0(C, E\ot \omega _C \ot L^\vee) = H^1(C, E^\vee \ot L)^\vee =0 ,$$
which implies that $$0\ge \chi (C, F\ot \omega _C \ot L^\vee ) = \deg F + \rank F (2g-2-\deg L)+\rank F (1-g) . $$
Hence, we conclude the proof.
\end{proof}

\begin{Def}\label{lowbd} Denote by $\langle E_0\rangle$ the smallest quiver {\em saturated} 
subsheaf of $E$ containing $E_0$ (which can be obtained by the intersection of all 
submodules $F$ satisfying $E_0\subset F$ and $E/F$ is torsion free).
Here the saturation means that the sheaf $E_i/\langle E\rangle _i$ is torsion free for every $i$.
\end{Def}

\begin{Prop}\label{asym} Fix $v=(v_i) \in \NN ^{Q_0}_{\ge 0}$ with $v_0=\dim K$ and let $d=(d_i)\in \ZZ ^{Q_0'}$. 
There is  a number $\delta_0 >0 $ such that for all $\delta \ge \delta _0$, 
the following conditions are equivalent for $\bf M$-twisted quiver bundles $E$ 
with numerical data $(v, d)$ in the category $\Repklo$.
\begin{enumerate}
\item $\delta$-semistability.
\item $\delta$-stability.
\item the stability as a $\bf M$-twisted quasimap to $X/\!\!/\!\!/_{\theta } G$ where $\theta =(1,...,1)$.
\item $\langle E_0\rangle = E$.
\end{enumerate}
\end{Prop}
\begin{proof} 
Let $v_1 := \sum _{i\in Q_0'} v_i$. If $v_1=0$, then the proof is obvious since in this case
any $\bM$-twisted sheaf is both $\delta$-stable and quasimap-stable.
Thus, we assume that $v_1\ne 0$.
  
 (3) $\Leftrightarrow$ (4): This follows from that $\theta$-stability of $E |_p$ for general $p\in C$
 if and only if $\langle E_0 \rangle |_p = E|_p$ for general $p\in C$
 if and only if $\langle E_0\rangle = E$.

For a $\delta$-semistable $\bM$-twisted quiver coherent sheaf $\{E_i\}$ with the numerical data
or for a $\bM$-twisted quasimap $(P, u)$ to $\XXX$ with the numerical data,
let $E:= \oplus _i E_i$ or $E:=\oplus _i P\times _G \CC ^{v_i}$ as a coherent sheaf of $\cO _C$-modules.
By the boundedness in Theorem \ref{bd} and Theorem \ref{twisted-symmetric}(1), 
such $E$'s are bounded. Therefore by Lemma \ref{bbd},
there is a number $m_0$ for which 
the condition $H^1(C, E\ot \phi ^*\mathcal{O}(m_0))=0$ in Lemma \ref{bbd} holds.
(In fact, such $E^\vee$'s are also bounded.)
Let $N=2(|\deg L|   + |1-g|)\rank E$
in Lemma \ref{bbd} with $L=\phi ^*\mathcal{O}(n_0)$.
Let $\Delta$ be a bound of $2|\deg \langle E_0 \rangle |$ which is finite since $E$ is bounded
and the construction of $\langle E_0\rangle $ from $E$ can be lifted in the level of family.

Denote $(-\infty, +\infty)$-valued slope functions $\mu _{(1)}$ and $\mu _{(2)}$ by
 $$\mu _{(1)}(v',d') = 2d'/v'_1 \ \ \text{ and } \ \  \mu _{(2)}(v', d') = 2v'_0/v'_1 \ \ \text{ for } v'_1 > 0 .$$
 Note that $2\mu_{\delta}(E)=\delta \mu _{(2)}(E)+\mu _{(1)}(E)$.
Take any number $\delta _0$ such that 
\begin{equation}\label{ineq}
\delta _0  \mathrm{Min}_{0\neq v'<v} |\mu _{(2)} (v) - \mu _{(2)} (v') | >  N+
 |\mu _{(1)} (E)| +\Delta, \end{equation} for any pair $v'=(v_0',v_1')$ of  integers satisfying
 $v_0'=v_0$ or $0$,  $0 < v_1'\le v_1$, and $v'\ne v$.

(1) $\Rightarrow$ (4): If $\langle E_0\rangle \neq E$, then $\mu _{(2)}(\langle E_0\rangle) > \mu _{(2)}(E)$,
which implies that $  \delta (\mu _{(2)} (E)-\mu _{(2)}(\langle E_0\rangle )) < - |\mu _{(1)} (E)| -\Delta$ by (\ref{ineq}).
Hence $ \mu _{\delta} (E) <\mu _{\delta} (\langle E_0\rangle )$. 

(3) $\Rightarrow$ (2): First, recall that for a point $p\in W \subset X$
the $\theta$-stabilities as a point in $W$ and  in $X$ coincide as seen in the proof of Proposition \ref{NC}.
Let $E'$ be a nonzero quiver subsheaf of $E$ and let $E$ be stable as a $\bM$-twisted quasimap. 
Since $\mu _{(2)}(E')\le \mu _{(2)}(E)$, which implies that $\mu _{\delta} (E') < \mu _{\delta} (E)$ 
by  (\ref{ineq}).
\end{proof}

\begin{Rmk}
Theorem \ref{bd} and Proposition \ref{asym} generalize the corresponding Diaconescu's works for ADHM case in \cite{Dia1, Dia}.
\end{Rmk}

\end{document}